\documentclass[10pt]{article}
\usepackage{amsmath, amsthm}\usepackage{enumerate}
\usepackage{amssymb}\usepackage{bbold}
\usepackage{color}
\newtheorem{theorem}{Theorem}[section]
\newtheorem{proposition}[theorem]{Proposition}

\newtheorem{example}[theorem]{Example}
\newtheorem{corollary}[theorem]{Corollary}

\newtheorem{definition}[theorem]{Definition}

\newcommand{\co}{\text{\rm co}}

\DeclareMathOperator{\convo}{\xrightarrow[]{o}}
\DeclareMathOperator{\convco}{\xrightarrow[]{c-o}}

\DeclareMathOperator{\convr}{\xrightarrow[]{ru}}
\DeclareMathOperator{\convcr}{\xrightarrow[]{c-ru}}

\DeclareMathOperator{\convn}{\xrightarrow[]{\|\cdot\|}}

\large
\makeatletter
\renewcommand{\subsection}{\@startsection{subsection}{1}
{0pt}{3.25ex plus 1ex minus.2ex}{-1em}{\normalfont\normalsize\bf}}
\makeatother
\begin{document}

\title{{\bf Collective order boundedness of sets of operators between ordered vector spaces}}
\maketitle
\author{\centering{{Eduard Emelyanov$^{1}$, Nazife Erkur\c{s}un-\"{O}zcan$^{2}$, Svetlana Gorokhova $^{3}$\\ 
\small $1$ Sobolev Institute of Mathematics, Novosibirsk, Russia\\ 
\small $2$ Hacettepe University, Ankara, Turkey\\
\small $3$ Southern Mathematical Institute, Vladikavkaz, Russia}
\abstract{It is proved that: each collectively order continuous set of operators from an Archimedean ordered vector space with 
a generating cone to an ordered vector space is collectively order bounded; and each collectively order to norm bounded set of operators 
from an ordered Banach space with a closed generating cone to a normed space is norm bounded.
Several applications to commutative operator semigroups on ordered vector spaces are given.}

\vspace{5mm}

\vspace{5mm}
}}

\large

\section{Introduction}

\hspace{4mm}
Collectively compact sets of operators between normed
spaces were introduced by Anselone and Palmer \cite{AP1968}.
Recently, collectively qualified sets of operators were investigated
in \cite{AEG2024,E0-2024,E1-2024,E2-2024,G2024}. In the present paper, we generalize 
some results of \cite{AS2005,E1996,E0-2024,TT2020} to collectively qualified sets in the 
ordered vector space setting. Results of the paper
are new even in the single operator case. For example, they imply that:
1) each order continuous operator on an Archimedean ordered vector space with a generating cone 
is order bounded; 2) each order to norm bounded operator on an ordered Banach space 
with a closed generating cone is continuous.

Throughout the text, vector spaces are real, operators are linear,
${\cal L}(X,Y)$ stands for the space of operators from a vector space $X$ to a vector space $Y$, 
and $x_\alpha\downarrow 0$ stands for a decreasing net in an ordered vector space such that $\inf\limits_\alpha x_\alpha=0$. 

Using of order convergence is common in study of collective properties in ordered vector spaces.
A net $(x_\alpha)$ in an ordered vector space $X$ is
\begin{enumerate}[-]
\item\ 
{\em order convergent} to $x\in X$ 
(briefly, o-convergent to $x$, or $x_\alpha\convo x$) whenever there exists 
a net $g_\beta\downarrow 0$ in $X$
such that, for each $\beta$ there is  $\alpha_\beta$ 
satisfying $-g_\beta\le x_\alpha-x\le g_\beta$ for all $\alpha\ge\alpha_\beta$. 
\item\ 
{\em relative uniform convergent} to $x\in X$ 
(briefly, ru-convergent to $x$, or $x_\alpha\convr x$) 
if, for some $u\in X_+$, there exists an increasing sequence 
$(\alpha_n)$ of indices satisfying $-\frac{1}{n}u\le x_\alpha-x\le\frac{1}{n}u$ for
all $\alpha\ge\alpha_n$. 
\end{enumerate}

We shall use the following version of \cite[Definition 1.2]{E0-2024}.
Let $X$ be an ordered vector space, $B$ a set, $A$ a directed set, and
${\cal B}=\{(x^b_\alpha)_{\alpha\in A}\}_{b\in B}$ a set of nets in $X$
indexed by $A$. We say that ${\cal B}$ 
\begin{enumerate}[-]
\item\ 
{\em collective \text{\rm o}-converges} to $0$ 
(briefly, ${\cal B}\convco 0$) 
if there exists a net $g_\beta\downarrow 0$ such that, for each $\beta$, 
there is $\alpha_\beta$ satisfying $\pm x^b_\alpha\le g_\beta$ for
$\alpha\ge\alpha_\beta$ and $b\in B$. 
\item\ 
{\em collective \text{\rm ru}-converges} to $0$ 
(briefly, ${\cal B}\convcr 0$)
if, for some $u\in X_+$, there exists an increasing sequence 
$(\alpha_n)$ of indices with $\pm x^b_\alpha\le\frac{1}{n}u$ for
$\alpha\ge\alpha_n$ and $b\in B$. 
\end{enumerate}
\noindent
Clearly, $x_\alpha\convo x\Longleftrightarrow\{(x_\alpha-x)\}\convco 0$,
and $x_\alpha\convr x\Longleftrightarrow\{(x_\alpha-x)\}\convcr 0$,
and if ${\cal B}\convco 0$ or ${\cal B}\convcr 0$ then each net of
${\cal B}$ eventually lies in $X_+-X_+$.

\begin{definition}\label{COCS and co}
{\em
(cf., \cite[Definition 2.1]{E0-2024})
Let ${\cal T}\subseteq{\cal L}(X,Y)$.\\ When $X$ and $Y$ are ordered vector spaces, the set ${\cal T}$ is
\begin{enumerate}[$a)$]
\item\ 
{\em collectively order bounded} (briefly, ${\cal T}\in{\bf L}_{ob}(X,Y)$) 
if, for every order interval $[a,b]$ in $X$ there exists an order interval $[u,v]$ in $Y$
such that ${\cal T}[a,b]:=\bigcup\limits_{T\in \mathcal T}{T}[a,b]\subseteq [u,v]$.
\item\ 
{\em collectively \text{\rm o}-continuous} (briefly, ${\cal T}\in\text{\bf L}_{oc}(X,Y)$) if 
${\cal T}x_\alpha\convco 0$ whenever $x_\alpha\convo 0$; 
\item\ 
{\em collectively \text{\rm ru}-continuous} (briefly, ${\cal T}\in\text{\bf L}_{ruc}(X,Y)$) if 
${\cal T}x_\alpha\convcr 0$ whenever $x_\alpha\convr 0$; 
\end{enumerate}
When $X$ is an ordered vector space and $Y$ an ordered normed space, the set ${\cal T}$ is 
\begin{enumerate}[$d)$]
\item\ 
{\em collectively almost order bounded} (briefly, ${\cal T}\in\text{\bf L}_{aob}(X,Y)$)
if, for every $[a,b]\subseteq X$ and $\varepsilon>0$, there exists 
$[u,v]\subseteq Y$ such that 
${\cal T}[a,b]\subseteq [u,v]+\varepsilon B_Y$.
\end{enumerate}
When $X$ is an ordered vector space and $Y$ is a normed space, the set ${\cal T}$ is 
\begin{enumerate}[$e)$]
\item\ 
{\em collectively order to norm bounded} (briefly, ${\cal T}\in\text{\bf L}_{onb}(X,Y)$) 
if the set ${\cal T}[a,b]$ is norm bounded for every $[a,b]\subseteq X$.
\end{enumerate}}
\end{definition}

\bigskip
The present paper is organized as follows.
Section 2 begins with Theorem~\ref{o-cont are obdd} saying that 
every collectively order continuous set of operators from an Archime\-dean ordered vector space with a generating cone to an ordered vector space
is collectively order bounded. Theorem~\ref{r-cont} describes relationship between 
collectively ru-continuous and collectively order bounded sets, and
Theorem~\ref{appl1} says that, under rather mild assumptions,
every collectively order to norm bounded set of operators is 
norm bounded. As a consequence, Corollary \ref{single r-cont = ob} says that 
each order to norm bounded operator from an ordered Banach space with a closed generating cone to a normed space is continuous.
Proposition \ref{prop-appl1} tell us that every order bounded operator between ordered Banach spaces with closed generating cones is continuous.
Section 3 is devoted to applications of obtained results to operator semigroups on ordered vector spaces. 

For the terminology and notations that are not explained 
in the text, we refer to \cite{AB2006,AT2007,E2007,K1985}.

\section{\large Main results}

\hspace{4mm}
We are going to prove the following results which generalize the corresponding ones obtained 
in \cite{AS2005,E0-2024,TT2020} to the setting of ordered vector spaces. 
\begin{enumerate}[-]
\item
Theorem \ref{o-cont are obdd} saying that every collectively order continuous set of operators from an Archi\-me\-dean ordered vector space 
with a generating cone to any ordered vector space is collectively order bounded. 
\end{enumerate}
\begin{enumerate}[-]
\item
Theorem \ref{r-cont} saying that collectively ru-continuous sets of operators between ordered vector spaces with generating cones agree with collectively order bounded sets.
\end{enumerate}
\begin{enumerate}[-] 
\item
Theorem \ref{appl1} saying that every collectively order to norm bounded set of operators from an ordered Banach space with a closed generating cone to a normed space norm bounded
\end{enumerate}

Let us begin with the following generalization of the first part of \cite[Theorem 2.2]{E0-2024} which is 
a collective version of the result by Abramovich and Sirotkin.
The proof is based on ideas in the proof of \cite[Theorem 2.1]{AS2005}.

\begin{theorem}\label{o-cont are obdd}
Let $X$ be an Archimedean ordered vector space with a generating cone. 
Then $\text{\bf L}_{oc}(X,Y)\subseteq\text{\bf L}_{ob}(X,Y)$ for every ordered vector space $Y$.
\end{theorem}

\begin{proof}
Let ${\cal T}\in\text{\bf L}_{oc}(X,Y)$. Since $X_+-X_+=X$, it is enough to prove 
order boundedness of ${\cal T}[0,b]$ for every $b\in X_+$.
Take $[0,b]\subseteq X$. Following \cite[Theorem 2.1]{AS2005},
consider the lexicographically directed set ${\cal I}=\mathbb{N}\times [0,b]$: 
$(m,z)\ge(n,z')$ if $m>n$ or $m=n$ and $z\ge z'$, and let $x_{(k,z)}=\frac{1}{k}z$.
Since $X$ is Archimedean, $\frac{1}{n}b\downarrow 0$.
It follows from $0\le x_{(m,z)}=\frac{1}{m}z\le\frac{1}{n}b$ for $(m,z)\ge(n,b)$
that $x_{(m,z)}\convo 0$.

Since ${\cal T}\in\text{\bf L}_{oc}(X,Y)$, 
there exists $g_\beta\downarrow 0$ in $Y$ such that, 
for every $\beta$ there exists $(m_\beta,y_\beta)$ satisfying
$\pm T x_{(m,y)}\le g_\beta$ for all $T\in{\cal T}$ and $(m,y)\ge(m_{\beta},y_{\beta})$.
Pick any $g_{\beta_0}$.  Since $(m_{\beta_0}+1,y)\ge(m_{\beta_0},y_{\beta_0})$, 
$$
   \pm T\Big(\frac{y}{m_{\beta_0}+1}\Big)=\pm Tx_{(m_{\beta_0}+1,y)}\le g_{\beta_0}\ \ \ 
   (y\in[0,b],T\in{\cal T}).
$$ 
Therefore, ${\cal T}[0,b]\subseteq[-(m_{\beta_0}+1)g_{\beta_0},(m_{\beta_0}+1)g_{\beta_0}]$.
\end{proof}

\noindent
In particular, this extends the Abramovich--Sirotkin theorem as follows.

\begin{corollary}\label{o-cont op is ob}
Each order continuous operator from an Archimedean ordered vector space with generating cone to any ordered vector space is order bounded.
\end{corollary}

\noindent
Note that the Archimedean assumption, which has being missed in \cite[Theorem 2.1]{AS2005},
is essential in Theorem \ref{o-cont are obdd} and in Corollary \ref{o-cont op is ob} even in the vector lattice setting.

\begin{example}\label{ex to o-cont op is ob}
Let $X$ be the lexicographically ordered real plane $\mathbb{R}^2$, 
$Y$ the coordi\-na\-te-wise ordered $\mathbb{R}^2$. Both $X$ and $Y$ are vector lattices.
The identity operator $T:X\to Y$ is not order bounded since
the order interval $[(0,0),(2,0)]$ of $X$ contains line $x=1$ that is not order
bounded in $Y$. However, $T$ is order continuous. To see this, it is enough 
to show that $(x_\alpha,y_\alpha)\downarrow 0$ in $X$ implies 
$T((x_\alpha,y_\alpha))=(x_\alpha,y_\alpha)\convo 0$ in $Y$. 
Suppose $(x_\alpha,y_\alpha)\downarrow 0$ in $X$ and let $\varepsilon>0$. 
There exists $\alpha_0$ with $x_\alpha=0$ for $\alpha\ge\alpha_0$.
Now, take $\alpha_1\ge\alpha_0$ 
such that $0\le\max\{x_\alpha,y_\alpha\}=y_\alpha\le\varepsilon$ 
for $\alpha\ge\alpha_1$. It follows that $T(x_\alpha,y_\alpha)\convo 0$ in $Y$.
\end{example}

An inspection of the proof of \cite[Theorem 2.1]{AS2005} shows that Abra\-mo\-vich and Sirotkin proved 
already that each \text{\rm ru}-continuous operator between vector lattices $E$ and $F$ is order bounded. Conversely, 
each order bounded operator $T:E\to F$ is \text{\rm ru}-continuous since, for every $u\in E_+$ there exists $w\in F$ 
such that $T[-u,u]\subseteq[-w,w]$, and hence $|x_\alpha|\le\frac{1}{n}u$ implies $|Tx_\alpha|\le\frac{1}{n}w$.
This assertion was independently proved by Taylor and Troitsky in \cite[Theorem 5.1]{TT2020}. 
Note that this fact can also be derived from the nonstandard criteria of order boundedness of operators
\cite[Theorem 1.7.2]{E1996} obtained by the first author. The following theorem generalizes this fact to the setting of ordered vector space.

\begin{theorem}\label{r-cont}
Let $X$ and $Y$ be ordered vector spaces.
\begin{enumerate}[$i)$]
\item 
If $X$ has a generating cone then $\text{\bf L}_{ruc}(X,Y)\subseteq\text{\bf L}_{ob}(X,Y)$.
\item 
If $Y$ has a generating cone then $\text{\bf L}_{ob}(X,Y)\subseteq\text{\bf L}_{ruc}(X,Y)$.
\end{enumerate}
\end{theorem}

\begin{proof}
$i)$ \ \
Let ${\cal T}\in\text{\bf L}_{ruc}(X,Y)$. Since $X_+-X_+=X$, to show ${\cal T}\in\text{\bf L}_{ob}(X,Y)$
it suffices to prove that ${\cal T}[0,b]$ is order bounded in $Y$ for each $b\in X_+$.
Take $b\in X_+$. Consider ${\cal I}=\mathbb{N}\times [0,b]$
and the net $x_{(k,z)}=\frac{1}{k}z$ like in the proof of Theorem \ref{o-cont are obdd}.
If $(m,z)\ge(n,b)$ in ${\cal I}$ then either $(m,z)=(n,b)$ or $m>n$, 
and hence $0\le x_{(m,z)}=\frac{1}{m}z\le\frac{1}{n}b$. Thus, $x_{(m,z)}\convr 0$ in $X$.

Since ${\cal T}\in\text{\bf L}_{ruc}(X,Y)$, there exist $u\in Y_+$ and an increasing 
sequence $(k_n,z_n)$ in ${\cal I}$ satisfying
$\pm T\big(\frac{z}{k}\big)=\pm Tx_{(k,z)}\le\frac{1}{n}u$ for
all $T\in{\cal T}$ and $(k,z)\ge(k_n,z_n)$. In particular,
$\pm T\big(\frac{z}{k_1+1}\big)=\pm Tx_{(k_1+1,z)}\le u$ for all $T\in{\cal T}$ and $z\in[0, b]$.
So, ${\cal T}[0,b]\subseteq[-(k_1+1)u,(k_1+1)u]$. 
It follows ${\cal T}\in\text{\bf L}_{ob}(X,Y)$.

\medskip
$ii)$ \ \
Let ${\cal T}\in\text{\bf L}_{ob}(X,Y)$ and $x_\alpha\convr 0$.
For some $u\in X_+$, there exists an increasing sequence 
$(\alpha_n)$ of indices such that $nx_\alpha\in[-u,u]$ for all $\alpha\ge\alpha_n$.
Since ${\cal T}\in\text{\bf L}_{ob}(X,Y)$ and $Y=Y_+-Y_+$ then ${\cal T}[-u,u]\subseteq[-w,w]$
for some $w\in Y_+$, and hence $n{\cal T}x_\alpha\subseteq[-w,w]$ for all $\alpha\ge\alpha_n$.
It follows ${\cal T}\in\text{\bf L}_{ruc}(X,Y)$.
\end{proof}

\begin{corollary}\label{r-cont = ob}
Order bounded operators between ordered vector spaces with generating cones agree with \text{\rm ru}-continuous operators.
\end{corollary}

Recall that order intervals in an ordered Banach space with a closed cone are norm bounded 
if and only if the cone is normal (cf., \cite[Theorem 2.40]{AT2007}). 
We include the following elementary collective extension of this fact.

\begin{proposition}\label{normal}
Let $X$ be an ordered vector space and $Y$ an ordered Banach space with closed generating normal cone.
Then $\text{\bf L}_{aob}(X,Y)\subseteq\text{\bf L}_{onb}(X,Y)$.
\end{proposition}

\begin{proof}
Let ${\cal T}\in\text{\bf L}_{aob}(X,Y)$ and $[a,b]\subseteq X$. 
Then ${\cal T}[a,b]\subseteq [-y,y]+B_Y$
for some $y\in Y_+$. Since $Y_+$ is normal, $[-y,y]+B_Y$
is norm bounded by \cite[Theorem 2.40]{AT2007}.
Therefore, ${\cal T}\in\text{\bf L}_{onb}(X,Y)$. 
\end{proof}

\noindent
The normality of positive cone in the range is essential in Proposition \ref{normal}. 

\begin{example}\label{example-(r-0x)}
Let $X=C^1[0,1]$ be an ordered Banach space of continuously differentiable functions on $[0,1]$.
The identity operator $I_X$ is norm and order bounded yet not order to norm bounded.
In particular, $I_X\in\text{\bf L}_{aob}(X,Y)\setminus\text{\bf L}_{onb}(X,Y)$.
\end{example}

The following result generalizes \cite[Theorem 2.1]{E0-2024} to ordered Banach spaces.

\begin{theorem}\label{appl1}
Let $X$ be an ordered Banach space with a closed generating cone
and $Y$ be a normed space. Then each ${\cal T}\in\text{\bf L}_{onb}(X,Y)$ is norm bounded.
\end{theorem}

\begin{proof}
Let ${\cal T}\in\text{\bf L}_{onb}(X,Y)$. Suppose, in contrary,
${\cal T}B_X$ is not norm bounded.
The Krein -- Smulian theorem (cf., \cite[Theorem 2.37]{AT2007})
implies $\alpha B_X\subseteq B_X\cap X_+-B_X\cap X_+$ for some $\alpha>0$,
and hence ${\cal T}(B_X\cap X_+)$ is not norm bounded. Pick sequences 
$(x_n)$ in $B_X\cap X_+$ and $(T_n)$ in ${\cal T}$ with $\|T_nx_n\|\ge n^3$ for all $n$, 
and set $x:=\|\cdot\|$-$\sum\limits_{n=1}^\infty n^{-2}x_n\in X_+$.
Since ${\cal T}\in\text{\bf L}_{onb}(X,Y)$ then ${\cal T}[0,x]\subseteq MB_Y$ for some $M<\infty$.
It follows from $n^{-2}x_n\in[0,x]$ that $\|T_n(n^{-2}x_n)\|\le M$ for every $n$.
This is absurd, because $\|T_n(n^{-2}x_n)\|\ge n$ for all $n$.
The proof is complete.
\end{proof}

\begin{corollary}\label{single r-cont = ob}
Every order to norm bounded operator from an ordered Banach space $X$ with a closed generating cone to a normed space $Y$ is continuous.
\end{corollary}

\noindent
It is worth noting that Corollary \ref{single r-cont = ob} generalizes the well known fact that every order bounded  
operator between Banach lattices is bounded. It is also known that every positive operator $T:X\to Y$ is continuous whenever $X$ is 
an ordered Banach space with a closed generating cone and $Y$ is an ordered Banach space with a closed cone, 
see Arendt \cite[Appendix]{A1984}, Aliprantis and Tourky \cite[Theorem 2.32]{AT2007}, or Arendt and Nittka \cite[Theorem 2.8]{AN2009}.

\medskip
Note that a positive operator need not to be order to norm bounded. Indeed, the identity operator $I_X$ on 
an ordered Banach space $X$ is order to norm bounded if and only if the positive cone $X_+$ is normal.
Therefore, the next proposition does not follow from Corollary \ref{single r-cont = ob}. 
Its proof below is a modification of the proof of \cite[Theorem 2.32]{AT2007}.

\begin{proposition}\label{prop-appl1}
Let $X$ and $Y$ be ordered Banach spaces with closed generating cones. Then every order bounded operator $T:X\to Y$ is continuous.
\end{proposition}

\begin{proof}
Let $T:X\to Y$ be order bounded. To establish that $T$ is continuous, it suffices to show that $T$ has a closed graph.
So, assume $x_n\convn 0$ in $X$ and $Tx_n\convn y$ in $Y$. We must show  that $y=0$.

According to \cite[Lemma 2.30]{AT2007}, there exist a subsequence $(n_k)$ such that $x_{n_k}\convr 0$. 
By Corollary \ref{r-cont = ob}, $T$ is \text{\rm ru}-continuous, and hence $Tx_{n_k}\convr 0$.
Thus, there exist some $u\in Y$ and a further subsequence $(n_{k_m})$ such that $\pm Tx_{n_{k_m}}\le \frac{1}{m}u$,
and hence $\frac{1}{m}u\pm Tx_{n_{k_m}}\in Y_+$ for each $m$. 
Since $Y_+$ is closed and $\frac{1}{m}u\pm Tx_{n_{k_m}}\convn\pm y$, we conclude that $\pm y\in Y_+$.
That is, y = 0 and the proof is finished.
\end{proof}

\noindent
We have the following collective extension of Proposition \ref{prop-appl1}.

\begin{proposition}\label{propa-appl1}
Let $X$ and $Y$ be ordered Banach spaces with closed generating cones, and $X^\sim\ne\{0\}$. The following assertions are equivalent.
\begin{enumerate}[$i)$]
\item\
Every ${\cal T}\in\text{\bf L}_{ob}(X,Y)$ is norm bounded.
\item\
The cone $Y_+$ is normal.
\end{enumerate}
\end{proposition}

\begin{proof}
$i)\Longrightarrow ii)$\
Assume that $Y_+$ is not normal. Then, there exists $y_0\in X_+$ such that 
the interval $[0,y_0]$ is not bounded. Pick $f_0\ne 0$ in $X^\sim$ and take $x_0\in X$ with $f_0(x_0)=1$.
Then, $\{f_0\otimes y\}_{y\in[0,y_0]}\in\text{\bf L}_{ob}(X,Y)$. However, the set $\{f_0\otimes y\}_{y\in[0,y_0]}$ is not norm bounded
because of $\bigcup\limits_{y\in[0,y_0]}(f_0\otimes y)(x_0)=[0,y_0]$. 

\medskip
$ii)\Longrightarrow i)$\
Let ${\cal T}\in\text{\bf L}_{ob}(X,Y)$. Then, ${\cal T}\in\text{\bf L}_{aob}(X,Y)$. So, ${\cal T}\in\text{\bf L}_{onb}(X,Y)$
by Proposition \ref{normal}. It follows from Theorem \ref{appl1} that ${\cal T}$ is norm bounded.
\end{proof}

The following example shows that the condition that $X_+$ is closed and generating is essential 
in Theorem \ref{appl1}, in Corollary \ref{single r-cont = ob}, and in Proposition \ref{prop-appl1}.

\begin{example}\label{ex to appl1}
Consider an infinite dimensional Banach space $X$.
\begin{enumerate}[$a)$]
\item \ 
Let $X_+:=\{0\}$ and take an arbitrary discontinuous operator $T$ on $X$. Then $T:(X,X_+)\to X$ is order to norm 
bounded and order bounded simply because order intervals in $X$ are singletons. Clearly $X_+$ is closed but not generating.
\item \ 
Pick a normalized Hamel basis $\{\xi:\xi\in\Xi\}$ of $X$, identify $X$ with $c_{00}(\Xi)$, and let 
$X_+:=\Big\{\sum\limits_{\xi\in\Xi}\lambda_\xi\xi: \lambda_\xi\ge 0\Big\}$. Then $X_+-X_+=X$ but
$X_+$ is not closed $($e.g., by \text{\rm \cite[Proposition 4.1]{AN2009}}$)$. 
Pick any sequence $(\xi_n)_{n=0}^\infty$ of distinct elements of $\Xi$ and define 
$T:(X,X_+)\to X$ by $T\xi=0$ for $\xi\in\Xi\setminus\{\xi_n:n\in\mathbb{N}\}$ and $T\xi_n=n\xi_0$ for $n\in\mathbb{N}$.
Clearly, $T$ is order to norm bounded and positive, yet not continuous.  
\end{enumerate}
\end{example}

\begin{example}\label{ex2 to appl1}
Let $X$ be an ordered vector space with a generating cone which satisfies the dominated decomposition property 
$($i.e., $[0,v_1]+[0,v_2]=[0,v_1+v_2]$ for every $v_1,v_2$ in $X_+$$)$, and let
$Y$ be a Dedekind complete vector lattice. 
It is well known that the ordered vector space ${\cal L}_{ob}(X,Y)$ of order bounded operators of ${\cal L}(X,Y)$ is a Dedekind 
complete vector lattice {\em (cf., \cite[Theorem 83.4]{Z1983})}. Let $A$ be an order bounded subset of ${\cal L}_{ob}(X,Y)$, say
$A\subseteq[-T,T]$ for some positive operator $T\in{\cal L}_{ob}(X,Y)$. 
Let $x\in X_+$. Since $|Sz|\le|S|z\le Tz\le Tx$ for all $z\in[0,x]$ and $S\in A$, then the set $A$ is collectively order bounded.

\medskip
From the other hand, not every collectively order bounded set of operators is order bounded. Indeed, consider $X=\ell^\infty$, $Y=\mathbb{R}$,
and $A=\{\delta^n: n\in\mathbb{N}\}\subset\ell^\infty_+$, where $\delta^n(\bold{x})=x_n$. The set $A$ is collectively order bounded, yet
there is no $f\in\ell^\infty_+$ with $A\subseteq[0,f]$, since otherwise $f(1,1,\ldots,1,\ldots)\ge\sum_{k=1}^n\delta^k(1,1,\ldots,1,\ldots)=n$
for every $n\in\mathbb{N}$ that is absurd.
\end{example}

\section{\large Some applications to operator semigroups}

\hspace{4mm}
Under a {\em semigroup} $(T_s)_{s\ge 0}$ on a vector space $X$ we understand  
a family of operators on $X$ satisfying $T_{s+t}=T_sT_t$ for all $s,t\in\mathbb{R}_+$, and $T_0=I_X$, where $I_X$ is the identity operator on $X$.
We shall also use the notation $(T_s)_{s\in J}$ for a one-parameter operator semigroup on $X$, 
where $J=\mathbb{R}_+, \mathbb{R}, \mathbb{Z}_+$, or $\mathbb{Z}$. Such a semigroup is always commutative. 
More generally, any commutative operator semigroup $\mathcal G$ on $X$ can be written
as $\mathcal G=(T_S)_{S\in J}$, where $T_S=S$ and $J=\mathcal G$ is directed by 
a partial order $RS\succeq S$ for all $R,S\in\mathcal G$.
The partial order $\succeq$ is extendable to the commutative operator semigroup
${\co}(\mathcal G)$ via letting $RS\succeq S$ for all $R,S\in{\co}(\mathcal G)$ 
(cf., \cite[p.75]{K1985}). 

\begin{proposition}\label{prop A-1}
Let $\mathcal G$ be a collectively order to norm bounded
commutative operator semigroup on an ordered Banach space $X$ with a closed generating cone containing a \text{\rm (}weakly\text{\rm )}
compact operator $S_0$. Then the semigroup $(S)_{{\co}(\mathcal G)\ni S\succeq S_0}$ 
is collectively \text{\rm (}weakly\text{\rm )} compact. In particular, $\{S\}_{S\succeq S_0}$
is collectively \text{\rm (}weakly\text{\rm )} compact.
\end{proposition}

\begin{proof}
By Theorem \ref{appl1}, $\|S\|\leq M<\infty$ for $S\in\mathcal G$, and hence
for all $S\in{\co}(\mathcal G)$. Thus, 
$$
   \bigcup_{{\co}(\mathcal G)\ni S\succeq S_0}SB_X=
   S_0\Big(\bigcup_{S\in{\co}(\mathcal G)}SB_X\Big)\subseteq S_0(MB_X)=MS_0B_X.
$$
Since $S_0B_X$ is relatively (weakly) compact, 
$(S)_{{\co}(\mathcal G)\ni S\succeq S_0}$ is collectively (weakly) compact. 
\end{proof}

\begin{theorem}\label{prop A-3}
Let $\mathcal G$ be a collectively almost order bounded commutative operator semigroup on an ordered Banach space $X$ 
with a closed generating cone, and assume that every order interval in $X$ is weakly compact. 
Then the net $(S)_{S\in\overline{{\co}}(\mathcal G)}$ converges strongly to a projection.
\end{theorem}

\begin{proof} 
Since order intervals in $X$ are weakly compact then every almost order bounded subset of $X$ is 
relatively weakly compact due to the Grothendieck result (cf., \cite[Theorem 3.44]{AB2006}).
In particular, $\mathcal G$ be a collectively order to norm bounded.
By Theorem \ref{appl1}, $\|S\|\leq M$ for some $M<\infty$ and all $S\in\mathcal G$. 

Let $x\in X_+$. Then, the set $\bigcup_{S\in\mathcal G}S[-x,x]$ is almost order bounded, and hence relatively weakly compact.
Since $(Sx)_{S\in\overline{{\co}}(\mathcal G)}$ lies in the weakly compact set $\overline{\co}(\bigcup_{S\in\mathcal G}S[-x,x])$, 
it has a weak cluster point, say $y$. Then $\lim\limits_{S\to\infty}\|Sx-y\|=0$ 
by the Eberlein theorem (cf., \cite[Theorem 1.5]{K1985}). Since $x\in X_+$ is arbitrary and $X_+$ is generating, 
$(S)_{S\in\overline{{\co}}(\mathcal G)}$ converges strongly to a projection.
\end{proof}

\begin{corollary}\label{prop A-3 + normal}
Let $(T_s)_{s\in J}$ be either a $C_0$-semigroup or  one-parameter discrete operator semigroup 
on an order continuous Banach lattice $E$. If $(T_s)_{s\in J}$ is collectively almost order bounded 
then $(T_s)_{s\in J}$ is mean ergodic. 
\end{corollary}

\noindent
The following example shows that the order continuity of the norm in $E$ is essential in Corollary \ref{prop A-3 + normal}.

\begin{example}
{\em (cf., \cite[Exercise 2.2.22]{E2007})}
Consider the operator $T:c\to c$ such that $T(a_1,a_2,\ldots a_k,\ldots)=\big(\frac{1}{2}a_1,\frac{2}{3}a_2,\ldots,\frac{k}{k+1}a_k,\ldots\big)$.
Clearly, $\big\{T^n\big\}_{n=1}^\infty\in\text{\bf L}_{ob}(c)$ for some $\varepsilon>0$. However,
the Ces\`aro averages $\frac{1}{n} \sum_{j=0}^{n-1}T^j$ do not converge in the norm at $(1,1,\ldots 1,\ldots)\in c$.
\end{example}

\noindent
The following corollary of Theorem \ref{prop A-3} is immediate, because order intervals 
in the predual of a Neumann algebra are weakly compact by the Akemann theorem (cf., \cite[Theorem 3.3.3]{E2007}).

\begin{corollary}\label{prop A-4 + normal}
Let $(T_s)_{s\in J}$ be either a $C_0$-semigroup or  one-parameter discrete operator semigroup 
on the predual $\mathcal M_*$ of a von Neumann algebra $\mathcal M$.
If $(T_s)_{s\in J}$ is collectively almost order bounded then $(T_s)_{s\in J}$ is mean ergodic. 
\end{corollary}

\medskip
Uniformly continuous semigroups is a starting point in the theory of $C_0$-semigroups. 
Here we discuss similar elementary notions concerning semigroups on ordered vector spaces.

\begin{definition}\label{locally qualified semigrops}
A semigroup $(T_s)_{s\ge 0}$ on an ordered vector space $X$ is 
\begin{enumerate}[$a)$]
\item 
{\em locally collectively order bounded} $($shortly, {\em \text{\rm lco}-bounded}$)$\\ 
whenever $\big\{T_s\big\}_{0\le s\le M}\in\text{\bf L}_{ob}(X)$ for every $M>0$.
\item 
{\em \text{\rm lco}-continuous} whenever $\big\{T_s\big\}_{0\le s\le M}\in\text{\bf L}_{oc}(X)$ for every $M>0$.
\item 
{\em \text{\rm lcru}-continuous} whenever $\big\{T_s\big\}_{0\le s\le M}\in\text{\bf L}_{ruc}(X)$ for every $M>0$.
\end{enumerate}
A semigroup $(T_s)_{s\ge 0}$ on a normed ordered vector space $X$ is 
\begin{enumerate}
\item[$d)$] 
{\em \text{\rm lcon}-bounded} if $\big\{T_s\big\}_{0\le s\le M}\in\text{\bf L}_{onb}(X)$ for every $M>0$.
\end{enumerate}
\end{definition}

\medskip
\noindent
In several important cases, the condition ``for every $M>0$" can be replaced 
by ``for some $\varepsilon>0$" in each of the four items of Definition \ref{locally qualified semigrops}.

\begin{proposition}\label{obz-lemma}
Let ${\cal T}=(T_s)_{s\ge 0}$ be a semigroup on an ordered vector space $X$ such that $X_+-X_+=X$. 
\begin{enumerate}[$i)$]
\item 
If $\big\{T_s\big\}_{0\le s\le\varepsilon}\in\text{\bf L}_{ob}(X)$ for some $\varepsilon>0$, then ${\cal T}$ is \text{\rm lco}-bounded.
\item 
If $\big\{T_s\big\}_{0\le s\le\varepsilon}\in\text{\bf L}_{ruc}(X)$ for some $\varepsilon>0$, 
then ${\cal T}$ is \text{\rm lcru}-continuous.
\item 
If $\big\{T_s\big\}_{0\le s\le\varepsilon}\in\text{\bf L}_{oc}(X)$ for some $\varepsilon>0$, then  ${\cal T}$ is \text{\rm lco}-continuous.
\item 
If $X$ is Banach space, $X_+$ is closed, and $\big\{T_s\big\}_{0\le s\le\varepsilon}\in\text{\bf L}_{onb}(X)$ for some $\varepsilon>0$, then
${\cal T}$ is \text{\rm lcon}-bounded.
\end{enumerate}
\end{proposition}

\begin{proof}
$i)$\
Let $\big\{T_s\big\}_{0\le s\le\varepsilon}\in\text{\bf L}_{ob}(X)$ for some $\varepsilon>0$
and fix $M\ge 0$. Let $n\varepsilon\ge M$, and take arbitrary $a\le b$ in $X$. 
It follows from $X_+-X_+=X$ that $[a,b]\subseteq[-a_0,a_0]$ for some $a_0\in X_+$.
Since $\big\{T_s\big\}_{0\le s\le\varepsilon}\in\text{\bf L}_{ob}(X)$ and $X_+$ is generating, 
$\bigcup\limits_{0\le s\le\varepsilon}{T_s}[-a_0,a_0]\subseteq[-a_1,a_1]$ for some $a_1\in X$.
Similarly, for some $a_2\in X$, we have 
$\bigcup\limits_{0\le s\le\varepsilon}{T_s}[-a_1,a_1]\subseteq[-a_2,a_2]$ and
$$
   \bigcup\limits_{\varepsilon\le s\le 2\varepsilon}{T_s}[-a_0,a_0]=
   T_\varepsilon\left(\bigcup\limits_{0\le s\le\varepsilon}{T_s}[-a_0,a_0]\right)\subseteq
   T_\varepsilon[-a_1,a_1]\subseteq[-a_2,a_2].
$$
It follows $\bigcup\limits_{0\le s\le 2\varepsilon}{T_s}[-a_0,a_0]\subseteq[-(a_1+a_2),a_1+a_2]$.
Repeating n-times the argument, we obtain $a_1,a_2,...,a_n\in X$ with 
$$
   \bigcup\limits_{0\le s\le M}T_s[a,b]\subseteq 
   \bigcup\limits_{0\le s\le n\varepsilon}{T_s}[-a_0,a_0]\subseteq
   \left[-\sum_{i=1}^{n}a_i,\sum_{i=1}^{n}a_i\right].
   \eqno(1)
$$
Since $[a,b]$ is arbitrary, (1) implies $\big\{T_s\big\}_{0\le s\le M}\in\text{\bf L}_{ob}(X)$, and since $M\ge 0$ is arbitrary, 
$(T_s)_{s\ge 0}$ is locally collectively order bounded.

\medskip
$ii)$\
By Theorem \ref{r-cont}, it follows from $i)$.

\medskip
$iii)$\
Suppose $\big\{T_s\big\}_{0\le s\le\varepsilon}\in\text{\bf L}_{oc}(X)$ for some $\varepsilon>0$,
and fix $M\ge 0$. Let $n\varepsilon\ge M$ and $x_\alpha\convo 0$ in $X$.
Pick a net $g^{(1)}_{{\beta_{(1)}}}\downarrow 0$ in $X$ such that, 
for each $\beta_{(1)}$ there is an $\alpha_{\beta_{(1)}}$ with 
$\pm T_sx_\alpha\le g^{(1)}_{{\beta_{(1)}}}$ for $0\le s\le\varepsilon$ and
$\alpha\ge\alpha_{\beta_{(1)}}$. Since $T_\varepsilon$ is order continuous,
$T_\varepsilon x_\alpha\convo 0$ in $X$. It follows from 
$\big\{T_s\big\}_{0\le s\le\varepsilon}\in\text{\bf L}_{oc}(X)$ that 
there exists a net $g^{(2)}_{{\beta_{(2)}}}\downarrow 0$ in $X$ such that, 
for each $\beta_{(2)}$ there is $\alpha_{\beta_{(2)}}$ with 
$\pm T_sT_\varepsilon x_\alpha\le g^{(2)}_{{\beta_{(2)}}}$ for $0\le s\le\varepsilon$ and
$\alpha\ge\alpha_{\beta_{(2)}}$. Repeating n-times, we obtain nets 
$$
   g^{(1)}_{{\beta_{(1)}}}\downarrow 0,\  
   g^{(2)}_{{\beta_{(2)}}}\downarrow 0,\ . . . \
   g^{(n)}_{{\beta_{(n)}}}\downarrow 0
$$
so that, for each $\beta_{(i)}$, $i=1,2,...,n$, there is $\alpha_{\beta_{(i)}}$ satisfying 
$$
   \pm T_sT^{i-1}_\varepsilon x_\alpha\le g^{(i)}_{{\beta_{(i)}}} \ \ \ \ 
   (0\le s\le\varepsilon \ \ \text{\rm and} \ \ \alpha\ge\alpha_{\beta_{(i)}}).
$$
Consider the net 
$g_{(\beta_{(1)},\beta_{(2)},...,\beta_{(n)})}=\sum_{i=1}^{n}g^{(i)}_{{\beta_{(i)}}}\downarrow 0$. 
Take an 
$$
   \alpha_{(\beta_{(1)},\beta_{(2)},...,\beta_{(n)})}\ge\alpha_{\beta_{(1)}},
   \alpha_{\beta_{(2)}},....,\alpha_{\beta_{(n)}}. 
$$
Then, for each $\alpha\ge\alpha_{(\beta_{(1)},\beta_{(2)},...,\beta_{(n)})}$,
$$
   \pm T_sT^{i-1}_\varepsilon x_\alpha\le g^{(i)}_{{\beta_{(i)}}}\le
   g_{(\beta_{(1)},\beta_{(2)},...,\beta_{(n)})} \ \ \ \ 
   (0\le s\le\varepsilon \ \ \text{\rm and} \ \ i=1,2,...,n).
$$
Therefore, for each $\alpha\ge\alpha_{(\beta_{(1)},\beta_{(2)},...,\beta_{(n)})}$,
$$
   \pm T_sx_\alpha\le g_{(\beta_{(1)},\beta_{(2)},...,\beta_{(n)})} \ \ \ \ 
   (0\le s\le n\varepsilon).
   \eqno(2)
$$
Since $x_\alpha\convo 0$ is arbitrary, (2) implies 
$\big\{T_s\big\}_{0\le s\le n\varepsilon}\in\text{\bf L}_{oc}(X)$, and hence
$\big\{T_s\big\}_{0\le s\le M}\in\text{\bf L}_{oc}(X)$.
Since $M\ge 0$ is arbitrary, $(T_s)_{s\ge 0}$ is locally collectively order continuous.

\medskip
$iv)$\
Suppose $\big\{T_s\big\}_{0\le s\le\varepsilon}\in\text{\bf L}_{onb}(X)$ for some $\varepsilon>0$
and fix $M\ge 0$. Take $n\in\mathbb{N}$ with $n\varepsilon\ge M$. Let $a\in X_+$.
By Theorem \ref{o-cont are obdd}, there exists $C\ge 0$ such that $\|T_s\|\le C$ for all
$0\le s\le\varepsilon$. Then
$$
   \sup\limits_{0\le s\le M}\|T_sa\|\le\sup\limits_{0\le s\le n\varepsilon}\|T_s\|\|a\|\le
   \Big(\sup\limits_{0\le s\le\varepsilon}\|T_s\|\Big)^n\|a\|\le C^n\|a\|.
   \eqno(3)
$$
Since $a\in X_+$ is arbitrary, (3) implies $\big\{T_s\big\}_{0\le s\le M}\in\text{\bf L}_{onb}(X)$.
As $M\ge 0$ is arbitrary, $(T_s)_{s\ge 0}$ is locally collectively order to norm bounded.
\end{proof}

Under rather mild assumptions, every order bounded (\text{\rm o}-continu\-ous, \text{\rm ru}-contin\-uous, order to norm bounded) 
operator between ordered Banach spaces gives a rise to \text{\rm lco}-bounded  (resp., \text{\rm lco}-bounded, \text{\rm lco}-continuous, 
\text{\rm lcru}-continuous, \text{\rm lcon}-bounded) operator semigroup $(T_t)_{t\ge 0}$. 

\begin{example}\label{example-(r-1x)}
Let $S:X\to Y$ be a bounded operator between Banach spaces.
Let $Z=X\oplus Y$, define a bounded operator $A:Z\to Z$ by $A(x,y)=(0,Sx)$,
and set $T_t=e^{tA}=\|\cdot\|$-$\sum_{k=0}^{\infty}\frac{t^kA^k}{k!}$ for $t\ge 0$. 
Then $(T_t)_{t\ge 0}$ is a uniformly continuous semigroup on the Banach space $Z$ such that
$T_t=I+tA$ for $t\ge 0$ because $A^2=0$.

Suppose $X$ and $Y$ are ordered Banach spaces and $X_+-X_+=X$.
\begin{enumerate}[$(a)$]
\item 
Let $Y_+$ be normal and $S:X\to Y$ be order bounded. 
Then $S$ is order to norm bounded since $S$ is order bounded and $Y_+$ is normal.
Theorem \ref{appl1} implies that $S$, and hence $A$ is bounded. Since $A$ is order bounded,
it follows from $T_t=I+tA$ for $t\ge 0$ that $(T_t)_{t\ge 0}$ is \text{\rm lco}-bounded.
\item 
Let $Y_+$ be normal and $S:X\to Y$ be \text{\rm ru}-continuous. By Theorem \ref{r-cont}, $S$ is order bounded, and 
hence $S$ is order to norm bounded because $Y_+$ is normal. Due to Theorem \ref{appl1},
$S$ is bounded. Thus $A$ is bounded. Since $A$ is \text{\rm ru}-continuous,
$T_t=I+tA$ implies $(T_t)_{t\ge 0}$ is \text{\rm lcru}-continuous.
\item
Suppose additionally that $X$ is Archimedean. Let $Y_+$ be normal and $S:X\to Y$ be order continuous. 
By Theorem \ref{o-cont are obdd}, $S$ is order bounded.
Since $Y_+$ is normal, $S$ is order to norm bounded.
Theorem \ref{appl1} implies boundedness of $S$, and hence of $A$. Since $A$ is order continuous
and $T_t=I+tA$, we have $(T_t)_{t\ge 0}$ is  \text{\rm lco}-continuous.
\item 
Let $S:X\to Y$ be order to norm bounded. Then, by Theorem \ref{appl1}, $S$ is bounded, and so is $A$. 
Since $A$ is order to norm bounded, $T_t=I+tA$ implies  $(T_t)_{t\ge 0}$ is \text{\rm lcon}-bounded.
\end{enumerate}
\end{example}

In recent works of Gl\"{u}ck, Kandi\'{c}, Kaplin, and Kramar Fijav\v{z} \cite{GK2024,KK2020,KK-F2020} relatively uniformly continuous 
semigroups on vector lattices were deeply studied.
Using of collectively qualified semigroups is too rough for getting progress in this direction. 
We postpone the investigation of relatively continuous semigroups on ordered vector spaces to a forthcoming paper.

{}
\end{document}